\documentclass[a4paper,reqno,10pt]{amsart}

\usepackage{amssymb}
\usepackage{amsmath}
\usepackage{amsthm}
\usepackage{enumitem}
\usepackage{xcolor}
\usepackage{graphicx}
\usepackage{url}
\usepackage{textcomp}
\usepackage{tikz}

\usepackage{fullpage}

\usepackage{todonotes}

\usepackage{texdraw,amstext,amsfonts,color,tabu,dsfont}

\usepackage{mathtools,tabularx}

\usepackage{float}

\definecolor{foge}{rgb}{0.1, 0.6, 0.1}

\numberwithin{equation}{section}

\newtheorem{theo}{Theorem}[section]

\newtheorem{prop}[theo]{Proposition}
\newtheorem{lem}[theo]{Lemma}
\newtheorem{cor}[theo]{Corollary}
\newtheorem{rem}[theo]{Remark}
\newtheorem{ex}[theo]{Example}

\theoremstyle{definition} 
\newtheorem{deff}[theo]{Definition}

\newcommand{\la}{\lambda}
\newcommand{\Pp}{\mathcal{P}}
\newcommand{\F}{\mathcal{F}}
\newcommand{\M}{\mathcal{M}}
\newcommand{\C}{\mathcal{C}}
\newcommand{\Zo}{\mathbb{Z}_{\geq 0}}

\title{A bijective proof of a generalization of the non-negative crank--odd mex identity}
\author{Isaac Konan}
\address{Universit\'e de Lyon, Universit\'e Claude Bernard Lyon 1, UMR5208, Institut Camille Jordan, F-69622 Villeurbanne, France}
\email{konan@math.univ-lyon1.fr}


\begin{document}
      
\begin{abstract}
Recent works of Andrews--Newman and Hopkins--Sellers unveil an interesting relation between two partition statistics, the crank and the mex. They state that, for a positive integer $n$, there are as many partitions of $n$ with non-negative crank as partitions of $n$ with odd mex. In this paper, we give a bijective proof of a generalization of this identity provided by Hopkins, Sellers and Stanton. Our method uses an alternative definition of the Durfee decomposition, whose combinatorial link to the crank was recently studied by Hopkins, Sellers, and Yee.  
\end{abstract}

\maketitle


\section{Introduction}


\subsection{State of art}

An \textit{integer partition} is a finite non-increasing sequence of positive integers. It then has the form $\la=(\la_1,\ldots,\la_s)$ with $\la_1\geq\cdots\geq \la_s\geq 1$. The terms $\la_i$ are called the parts of $\la$, and we denote by $\ell(\la)=s$ and $|\la|=\la_1+\cdots+\la_s$ respectively the length and the weight of the partition $\la$. For example, $\ell(\emptyset)=|\emptyset|=0$. For a non-negative integer $n$, an integer partition with weight $n$ is commonly called a partition of $n$. For example, the partition of $3$ are $(3),(2,1)$ and $(1,1,1)$, respectively with length $1,2$ and $3$. Let $\Pp$ be the set of integer partitions. For $\F\subset \Pp$, $\overline{\F}$ is the complementary of $\F$, i.e. $\overline{\F} = \{\la \in \Pp:\la \notin \F\}$. In the remainder of the paper, the term ``partition'' stands for an integer partition, and $\sharp A$ denotes the number of elements in the set $A$.

In a 1988 paper \cite{AG88}, Andrews and Garvan formally provide a definition of Dyson's crank, a partition statistic introduced by Dyson in \cite{D44} to combinatorially explain a divisibility property of partitions.  
Let $\la$ be a partition. Set $\omega(\la)=\sharp\{i\in \{1,\ldots,\ell(\la)\}:\la_i=1\}$, the number of occurrences of $1$ as part of $\la$, and set $\eta(\la)=\sharp\{i \in \{1,\ldots,\ell(\la)\} :\la_i>\omega(\la)\}$ the number of parts greater than $\omega(\la)$. The crank of $\la$, denoted $crank(\la)$, is defined by the relation
\begin{equation}
crank(\la) = \begin{cases}
             \la_1  \ \ \ \ \ \ \ \ \ \ \ \ \ \ \text{if} \ \ \omega(\la)=0\,,\\
             \eta(\la)-\omega(\la)\ \ \text{if} \ \ \omega(\la)>0\,.
             \end{cases}
\end{equation}
In particular, set $crank(\emptyset)=0$. One can easily check that $-|\la|\leq crank(\la)\leq |\la|$.
Given integers $m,n$ with $-n\leq m\leq n$, $C(m,n)$ denotes the number of partitions of weight $n$ and crank $m$, with the exception that $C(1,1)=-C(0,1)=1$. In \cite{G88}, Garvan explicitly provides the generating function for the crank.

\begin{theo}[Garvan]\label{theo:Garvan} We have 
\begin{equation}
(x-1)q+\sum_{\la\in \Pp} x^{crank(\la)}q^{|\la|}=\sum_{n\geq 0}\sum_{m=-n}^n C(m,n)x^mq^n= \frac{(q;q)_\infty}{(qx,qx^{-1};q)_\infty}
\end{equation}
where $(a_1,\ldots,a_t;q)_\infty = \prod_{k\geq 0}\prod_{i=1}^t(1-a_iq^k)$.
\end{theo}

\begin{cor}\label{cor:cranksym}
Given integers $n\geq m\geq 0$, we have $C(m,n)=C(-m,n)$.
\end{cor} 
  
Recent works involved the use of a new partition statistic, the \textit{minimal excludant} or \text{mex}, defined as the smallest positive integer which is not a part of the partition.
For $\la\in \Pp$, set $mex(\la)$ to be the mex of $\la$. For example, $mex(\emptyset)=mex((5,3,2,2))=1$. A curious yet interesting connection between the mex and the crank arose from the works of Andrews--Newman \cite{AN20} and Hopkins--Sellers \cite{HS20}, who independently found the following result. 

\begin{theo}\label{theo:crank-mex}
At fixed weight, the number of partitions with non-negative crank is equal to the number of partitions with odd mex.
\end{theo}

In \cite{HSS22}, Hopkins, Sellers and Stanton provide a broad generalization of Theorem \ref{theo:crank-mex} by introducing a notion generalizing the mex. For $j\geq 0$, and $\la$ a partition containing the part $j$, $mex_j(\la)$ is the smallest integer greater than $j$ which is not a part of $\la$. 

\begin{theo}\label{theo:main}
For $j\in \Zo$, at fixed weight greater than $1$, the number of partitions $\la$ with a part $j$ such that $mex_j(\la)-j$ is odd is equal to the number of partitions with crank at least equal to $j$.
\end{theo}

The case $j=0$ of the above theorem implies Theorem \ref{theo:crank-mex} as $mex_0(\la)=mex(\la)$ and $0$ can be seen as a fictitious part of all partitions.

The aim of this paper is to provide a purely bijective proof Theorem \ref{theo:main}. This generalization will derive from a key result related the Durfee decomposition of partitions.
 

\subsection{Statement of results}

We first extend the generalization of the notion of mex to all partitions.

\begin{deff}\label{def:j-mex}
Let $j\in \Zo$. For $\la\in \Pp$, the $j$-mex of $\la$, denoted $mex_j(\la)$, is the smallest integer greater than $j$ which is not a part of $\la$. 
For example, we have $mex_j(\emptyset)=j+1$ for $j\in \Zo$, 
$$mex_0(5,3,2,2)=1, mex_1(5,3,2,2)=mex_2(5,3,2,2)=mex_3(5,3,2,2)=4, mex_4(5,3,2,2)=6\,,$$ 
and finally $mex_j(5,3,2,2)=j+1$ for $j\geq 5$. Denote by $\M_{j}$ the set of partitions with a $j$-mex parity different from $j$. With the previous examples, we have that $\emptyset\in M_j$ for $j\in \Zo$, and  $$(5,3,2,2)\in \M_0\cap\M_1\cap\overline{\M}_2\cap \M_3 \cap\overline{\M}_4 \cap \bigcap_{j\geq 5}\M_j\,.$$
\end{deff}

In the remainder of the paper, for all $\la\in \Pp$, we  set $\la_0=\infty$ and $\la_{\ell(\la)+1}=0$, so that the sequence $(\la_0,\cdots,\la_{\ell(\la)+1})$ remains non-increasing. A partition then becomes a non-increasing sequence starting from $\infty$ and ending by $0$. For example, the partition $\emptyset$ is associated to the sequence $(\infty,0)$ with $\emptyset_0=\infty$ and $\emptyset_1=0$. For $j\in \Zo$, denote by $\Pp_j$ the set of partitions which do not have  $j$ as part. Conversely, $\overline{\Pp}_j$ is the set of partitions with a part $j$. A rephrasing of Theorem \ref{theo:main} is then the following.

\begin{theo}\label{theo:main2}
For $j\in \Zo$, at fixed weight greater than $1$, the number of partitions in $\M_j\cap \overline{\Pp}_j$ is equal to the number of partitions with crank at least equal to $j$.
\end{theo}

The bijective proof of Theorem \ref{theo:main2} that we provide in this paper, was deeply inspired by the work of Hopkins, Sellers, and Yee, who described in \cite{HSY} combinatorial relations that link the crank and the Durfee decomposition of a partition. Our work is based on a simple yet subtle  definition related the very notion of Durfee decomposition. 

\begin{deff}
Let $\la\in \Pp$. The function $i \mapsto \la_i-i$ is then decreasing on $\{0,\ldots,\ell(\la)+1\}$, with $\la_0-0=\infty$ and $\la_{\ell(\la)+1}-(\ell(\la)+1)<0$. Therefore, for $j\in \Zo $, there exists a unique integer $d_j^\la\in \{0,\ldots,\ell(\la)\}$ such that $\la_{d_j^\la}-d_j^\la\geq j >\la_{d_j^\la+1}-(d_j^\la+1)$.  Formally written, 
\begin{equation}\label{eq:dj}
d_j^\la = \max \{i\in\{0,\ldots,\ell(\la)\}: \la_i-i\geq j\}\,.
\end{equation}
For example, $d_j^\emptyset=0$, and for $\la=(5,3,2,2)$,
$$d_0^\la = d_1^\la = 2, \, d_2^\la = d_3^\la = d_4^\la = 1 \text{ and }d_j^\la = 0 \text{ for }j\geq 5\,.$$
We also denote by $\F_j$ the set of partitions $\la$ such that $j\notin \{\la_i-i: i \in \{1,\ldots,\ell(\la)\}\}$, which is equivalent to saying that $\la_{d_j^\la}-d_j^\la> j >\la_{d_j^\la+1}-(d_j^\la+1)$. Conversely, $\overline{\F}_j$ is the set of partitions $\la$ satisfying $\la_{d_j^\la}=d_j^\la+ j$.
\end{deff}

\begin{rem}\label{rem:dj}
For $j\in \Zo$, as $\la_{d_{j+1}^\la+2}-(d_{j+1}^\la+2)<j<\la_{d_{j+1}^\la}-d_{j+1}^\la$ when $d_{j+1}^k<\ell(\la)$, we always have that $d_{j}^\la-d_{j+1}^\la\in \{0,1\}$. Therefore, the sequence $(d_j^\la)_{j\geq 0}$ is non-increasing while the sequence $(d_j^\la+j)_{j\geq 0}$ is non-decreasing.
\end{rem}

\begin{rem}\label{rem:durfee}
The Durfee decomposition of a partition $\la=(\la_1,\ldots,\la_{\ell(\la)})$ is the triplet $(d_0^\la,\mu,\nu)$ with $(\mu,\nu)=[(\mu_1,\ldots,\mu_{d_0^\la}),(\nu_1,\ldots,\nu_{d_0^\la})]$ such that, for all $i\in \{1,\ldots,d_0^\la\}$, $\mu_i=\la_i-i$ and $\nu_i = \sharp\{u\in\{1,\ldots,\ell(\la)\}: \la_u\geq i\} - i$. Inversely, for any triplet $(t,\mu,\nu)$ such that $\mu$ and $\nu$ are increasing sequences of $t$ non-negative integers, we associate the partition $\la$ with length $\nu_1+1$ and 
$$
\begin{cases}
\la_i=\mu_i+i \ \ \ \ \ \ \ \ \ \ \ \ \ \ \ \text{for} \ \ 1\leq i\leq t \\
\la_i=\sharp\{u: \nu_u+u\geq i\} \ \ \text{for} \ \ t+1\leq i\leq \nu_1+1\,. 
\end{cases}
$$
We note $\la \equiv (t,\mu,\nu)$, and one can check that $|\la|=t+\sum_{u=1}^t \mu_u+\nu_u$. In particular, $\emptyset\equiv(0,\emptyset,\emptyset)$.

Observe that, for $j\in \Zo$, $\F_j$ can be equivalently defined as the set of partitions $\la\equiv (d_0^\la,\mu,\nu)$ such that $j$ is not in $\mu$.
\end{rem}

We now provide an intermediate result that plays a fundamental role in the bijective proof of Theorem \ref{theo:main}. 
 
\begin{theo}\label{theo:mainbis}
For $j\in \Zo$, at fixed weight, the number of partitions of $\M_{j}$ is equal to the number of partitions in $\F_j$.
\end{theo}

\begin{cor}\label{cor:mainbis}
For $j\in \Zo$, at fixed weight, the number of partitions of $\overline{\M}_{j}$ is equal to the number of partitions in $\overline{\F}_j$.
\end{cor}

Using the following reformulation of a result provided by Hopkins, Sellers and Stanton in \cite{HSS22}, we derive the generalization of Theorem \ref{theo:crank-mex} from Theorem \ref{theo:mainbis} and Corollary \ref{cor:cranksym}. 

\begin{theo}\label{theo:crank-frobenius}
Let $j\in \Zo$. Then, at fixed weight greater than $1$, there are as many partitions in $\F_j\cap \overline{\Pp}_j$ as partitions with crank at most equal to $-j$.
\end{theo}

By adding a part $j$ to the partitions in $\M_j$ and $\F_j$ when $j>0$, Theorem \ref{theo:crank-frobenius} implies that, at fixed weight, there are as many partitions in $\M_j\cap\overline{\Pp}_j$ as partitions in $\F_j\cap\overline{\Pp}_j$. Then, by Theorem \ref{theo:mainbis}, at fixed weight greater than $1$, there are as many partitions in $\M_j\cap\overline{\Pp}_j$ as partitions with crank at most equal to $-j$. We finally obtain Theorem \ref{theo:main} from Corollary \ref{cor:cranksym}.

The remainder of the paper is organized as follows. We first provide in Section \ref{sec:analyticproof} a simple analytic proof of Theorem \ref{theo:mainbis}. Then, in Sections \ref{sec:mainbijection} and \ref{sec:proofmain}, a bijection for Theorem \ref{theo:mainbis} is given, so as the proof of its well-definedness. After that, in Section \ref{sec:corbijection}, we give a direct bijective proof of Corollary \ref{cor:mainbis} in the spirit of the bijection of Section \ref{sec:mainbijection}.
In Section \ref{sec:crank-frobenius}, Theorem \ref{theo:crank-frobenius} is proved bijectively. Finally, in Section \ref{sec:conclusion}, we provide the full scope of bijective proof of Theorem \ref{theo:main} with a bijection for Corollary \ref{cor:cranksym} given in Section \ref{sec:cor}.   


\section{Analytic proof of Theorem \ref{theo:mainbis}}\label{sec:analyticproof}

For $j,k\in \Zo$, set $\Delta_{j,k} = (j+k,\ldots,j+1)$ the partition consisting of $k$ consecutive integers ending by $j+1$, and  $\Delta_{j,0}=\emptyset$. Then, $|\Delta_{j,k}| = \frac{k(k+1)}{2}+jk$.
Hence, by Definition \ref{def:j-mex}, the set of partitions with $j$-mex equal to $j+k+1$ can be associated to $\{\Delta_{j,k}\}\times \Pp_{j+k+1}$, i.e. $mex_j(\la)=j+k+1$ if and only if there exists a unique partition, without a part $j+k+1$, whose parts are exactly those of $\la$  except once the parts $j+1,\ldots,j+k$.
We then have 
$$
\sum_{\la\in \Pp} x^{mex_j(\la)}y^{\ell(\la)}q^{|\la|} = \frac{x^{j+1}}{(qy;q)_\infty}\sum_{k\geq 0} (xy)^{k}(1-yq^{j+k+1})q^{|\Delta_{j,k}|}\,.
$$ 
Using the above relation with $x=y=1$ and the sum run over $k$ even, we have
\begin{align*}
\sum_{\la \in \M_{j}} q^{|\la|} &= \frac{1}{(q;q)_\infty}\sum_{k\geq 0} (1-q^{j+2k+1})\cdot q^{|\Delta_{j,2k}|}\\
&= \frac{1}{(q;q)_\infty}\sum_{k\geq 0} (-1)^k q^{|\Delta_{j,k}|}\\
&= \frac{1}{(q;q)_\infty}\sum_{k\geq 0} (-1)^k q^{\frac{k(k+1)}{2}+jk}\,.
\end{align*}
Also, by the Jacobi triple product, 
$$(-x;q)_\infty (-x^{-1}q;q)_\infty=\frac{1}{(q;q)_\infty}\cdot\sum_{k\in \mathbb{Z}} (-1)^k x^k q^{\frac{k(k-1)}{2}}\,,$$
and by Remark \ref{rem:durfee},
\begin{align*}
\sum_{\la \in \F_j} q^{|\la|} &= [x^0]\left[(-x;q)_\infty (-x^{-1}q;q)_\infty\right]\cdot\frac{1}{(1+xq^j)}\\
&= [x^0]\left[\frac{1}{(q;q)_\infty}\cdot\sum_{k\in \mathbb{Z}} (-1)^k x^k q^{\frac{k(k-1)}{2}}\right]\cdot \left(\sum_{k\geq 0} (-1)^kx^kq^{jk}\right)\\
&= \frac{1}{(q;q)_\infty}\cdot \sum_{k\geq 0} (-1)^k q^{\frac{k(k+1)}{2}+kj}\,.
\end{align*}
In conclusion, 
$\sum_{\la \in \M_j} q^{|\la|}=\sum_{\la \in \F_j} q^{|\la|}$.


\section{Bijection for Theorem \ref{theo:mainbis}}\label{sec:mainbijection}

In this section, we provide two inverse maps for the bijective proof of Theorem \ref{theo:mainbis}. In our bijections, we will apply some transformations on pairs of partitions in $ \bigsqcup_{k\geq 0} \{\Delta_{j,k}\}\times \Pp$. One has to keep in  mind that $\Pp$ can  be trivially associated to $\{\Delta_{j,0}\}\times \Pp = \{\emptyset\}\times \Pp$. Finally, we identify $\M_j$ to $ \bigsqcup_{k\geq 0} \{\Delta_{j,2k}\}\times \Pp_{j+2k+1}$ and $\F_j$ to $\{\Delta_{j,0}\}\times \F_j$. Two maps will then be constructed,
 $$\Phi_j : \bigsqcup_{k\geq 0} \{\Delta_{j,2k}\}\times \Pp_{j+2k+1}\to \{\Delta_{j,0}\}\times \F_j$$
 and 
 $$\Psi_j : \{\Delta_{j,0}\}\times \F_j \to \bigsqcup_{k\geq 0} \{\Delta_{j,2k}\}\times \Pp_{j+2k+1}\,,$$
in such a way that the partitions in the pairs have their weight and their length sums conserved during the process. By abuse of notation, if $\la\in \M_j$ is identified to the pair $(\Delta_{j,2k},\mu)$, we write $\Phi_j(\la)=\nu\in \F_j$ such that $\Phi_j((\Delta_{j,2k},\mu))=(\Delta_{j,0},\nu)$. The same convention stands for $\Psi_j:\F_j\mapsto \M_j$. 

\subsection{From $\bigsqcup_{k\geq 0} \{\Delta_{j,2k}\}\times \Pp_{j+2k+1}$ to $\{\Delta_{j,0}\}\times \F_j$} \label{sec:Phi}

Let $\phi_j$ be the map defined on 
$$\left(\bigsqcup_{k\geq 0} \{\Delta_{j,2k}\}\times \Pp\right)\setminus \left(\{\Delta_{j,0}\}\times \F_{j}\right)=\left(\bigsqcup_{k\geq 1} \{\Delta_{j,2k}\}\times \F_{j+2k}\right) \sqcup \left(\bigsqcup_{k\geq 0} \{\Delta_{j,2k}\}\times \overline{\F}_{j+2k}\right) $$ as follows.

\begin{enumerate}

\item For $k\geq 0$ and $(\Delta_{j,2k},\la)\in \{\Delta_{j,2k}\}\times \overline{\F}_{j+2k}$, do the transformation 
$$\la_1,\ldots,\la_{d_{j+2k}^\la}  \mapsto \la_1+1,\ldots,\la_{d_{j+2k}^\la-1}+1, 1+j+2k\,.$$
This means that, in $\la$, we delete the part $\la_{d_{j+2k}^\la}=d_{j+2k}^\la+j+2k$, add $1$ to the $d_{j+2k}^\la-1$ largest \textit{finite} parts and add a part $1+j+2k+1$. We then obtain a partition $\mu$, and set $\phi_j((\Delta_{j,2k},\la))=(\Delta_{j,2k},\mu)$. Observe that $|\mu|=|\la|$ and $\ell(\mu)=\ell(\la)$, so that the weight and length sums are conserved. Moreover, the transformation does not involved parts less than $1+j+2k$, so that the parts at most equal to $j$ are conserved.

\item For $k\geq 1$ and $(\Delta_{j,2k},\la)\in \{\Delta_{j,2k}\}\times \F_{j+2k}$, do the transformation 
$$\Delta_{j,2k}, \la_1,\ldots,\la_{d_{j+2k}^\la}\mapsto\Delta_{j,2k-2}, \la_1-1,\ldots,\la_{d_{j+2k}^\la}-1,d_{j+2k}^\la+j+2k,j+2k-1\,.$$
This means that, in $\Delta_{j,2k}$, we deleted the parts $j+2k,j+2k-1$, and in $\la$, we subtract $1$ to  the $d_{j+2k}^\la$ largest finite parts and add the parts $d_{j+2k}^\la+j+2k$ and $j+2k-1$ to obtain a partition $\mu$. We finally set $\phi_j((\Delta_{j,2k},\la))=(\Delta_{j,2k-2},\mu)$. Note that $|\mu|=|\la|+(j+2k)+(j+2k-1)$, and $\ell(\mu)=\ell(\la)+2$, so that the weight and length sums are conserved. In addition, The parts at most equal to $j$ are conserved.

\end{enumerate}

For all $(\Delta_{j,2k},\la)\in \{\Delta_{j,2k}\}\times \Pp_{j+2k+1}$, 
iterate the map $\phi_j$ as long  as it is possible. The iteration stops as soon as we reach a pair in $\{\Delta_{j,0}\}\times \F_j$. We then set $\Phi_j((\Delta_{j,2k},\la)) = \phi_j^{u}((\Delta_{j,2k},\la)) \in \{\Delta_{j,0}\}\times \F_j$. We finally observe that the transformations occur only on parts greater than $j$. Therefore, if $\Phi_j$ is well-defined, it then induces a map from $\bigsqcup_{k\geq 0} \{\Delta_{j,2k}\}\times\left(\Pp_{j+2k+1}\cap\overline{\Pp}_j\right)$ to $\{\Delta_{j,0}\}\times \left(\F_j\cap \overline{\Pp}_j\right)$. 

\begin{ex}
For $j,k\in\Zo$, 
$$\Phi_j((\Delta_{j,2k},\emptyset))=(\Delta_{j,0},\Delta_{j,2k})\,,$$
as $\phi_j^u((\Delta_{j,2k},\emptyset))=(\Delta_{j,2k-2u},\Delta_{j+2k-2u,2u})$ for $0\leq u\leq k$.
\end{ex}

\begin{ex}\label{ex:main}
Consider the partition $(11,8,7,7,5,5,4,3,2,2)$. It belongs to $\M_j$ for $j\neq 2,4,7,10$, and the corresponding pairs are respectively
\small{
$$(\Delta_{0,0},(11,8,7,7,5,5,4,3,2,2)),(\Delta_{1,4},(11,8,7,7,5,2)),(\Delta_{3,2},(11,8,7,7,5,3,2,2))\,,$$
$$(\Delta_{5,0},(11,8,7,7,5,5,4,3,2,2)),(\Delta_{6,2},(11,7,5,5,4,3,2,2)),(\Delta_{j,0},(11,8,7,7,5,5,4,3,2,2)) \text{ for } 10 \neq j\geq 8\,.$$
}
We now represent in tables the different iterations of each $\phi_j$ for $j\in\{ 0,1,3,5\}$.
For $j=0$, 
$$
\begin{array}{|c|c|c|c|c|c|}
\hline 
\text{Iteration}&k&\la&d_{2k}^\la&\text{case}&\phi_0(\la)\\
\hline\hline
1&0&(11,8,7,7,5,5,4,3,2,2)&5&(1)&(12,9,8,8,5,4,3,2,2,1)\\
2&0&(12,9,8,8,5,4,3,2,2,1)&5&(1)&(13,10,9,9,4,3,2,2,1,1)\\
3&0&(13,10,9,9,4,3,2,2,1,1)&4&-&-\\
\hline
\end{array}
$$
and \small{$\Phi_0(\Delta_{0,0},(11,8,7,7,5,5,4,3,2,2))= \phi_0^2(\Delta_{0,0},(11,8,7,7,5,5,4,3,2,2))= (\Delta_{0,0},(13,10,9,9,4,3,2,2,1,1))$}. For $j=1$, 
$$
\begin{array}{|c|c|c|c|c|c|}
\hline 
\text{Iteration}&k&\la&d_{1+2k}^\la&\text{case}&\phi_1(\la)\\
\hline\hline
1&2&(11,8,7,7,5,2)&2&(2)&(10,7,7,7,7,5,4,2)\\
2&1&(10,7,7,7,7,5,4,2)&4&(1)&(11,8,8,7,5,4,4,2)\\
3&1&(11,8,8,7,5,4,4,2)&4&(1)&(12,9,9,5,4,4,4,2)\\
4&1&(12,9,9,5,4,4,4,2)&3&(2)&(11,8,8,6,5,4,4,4,2,2)\\
5&0&(11,8,8,6,5,4,4,4,2,2)&4&-&-\\
\hline
\end{array}
$$
and $\Phi_1((\Delta_{1,4},(11,8,7,7,5,2)))= \phi_1^4((\Delta_{1,4},(11,8,7,7,5,2)))= (\Delta_{1,0},(11,8,8,6,5,4,4,4,2,2))$. For $j=3$
$$
\begin{array}{|c|c|c|c|c|c|}
\hline 
\text{Iteration}&k&\la&d_{3+2k}^\la&\text{case}&\phi_3(\la)\\
\hline\hline
1&1&(11,8,7,7,5,3,2,2)&2&(2)&(10,7,7,7,7,5,4,3,2,2)\\
2&0&(10,7,7,7,7,5,4,3,2,2)&4&(1)&(11,8,8,7,5,4,4,3,2,2)\\
3&0&(11,8,8,7,5,4,4,3,2,2)&4&(1)&(12,9,9,5,4,4,4,3,2,2)\\
4&0&(12,9,9,5,4,4,4,3,2,2)&3&-&-\\
\hline
\end{array}
$$
and $\Phi_3((\Delta_{3,2},(11,8,7,7,5,3,2,2)))=\phi_3^3((\Delta_{3,2},(11,8,7,7,5,3,2,2)))=(\Delta_{3,0},(12,9,9,5,4,4,4,3,2,2))$. For $j= 5$,
$$
\begin{array}{|c|c|c|c|c|c|}
\hline 
\text{Iteration}&k&\la&d_{5+2k}^\la&\text{case}&\phi_5(\la)\\
\hline\hline
1&0&(11,8,7,7,5,5,4,3,2,2)&(2)&-&-\\
\hline
\end{array}
$$
and $\Phi_5(\Delta_{5,0},(11,8,7,7,5,5,4,3,2,2))= \phi_5^0(\Delta_{5,0},(11,8,7,7,5,5,4,3,2,2)) =(\Delta_{5,0},(11,8,7,7,5,5,4,3,2,2))$.
\end{ex}

\subsection{From $\{\Delta_{j,0}\}\times \F_j$ to $\bigsqcup_{k\geq 0} \{\Delta_{j,2k}\}\times \Pp_{j+2k+1}$}\label{sec:Psi}

Let $\psi_j$ be the map defined on 
$$\left(\bigsqcup_{k\geq 0} \{\Delta_{j,2k}\}\times \Pp\right)\setminus \left(\bigsqcup_{k\geq 0} \{\Delta_{j,2k}\}\times \Pp_{j+2k+1}\right)= \bigsqcup_{k\geq 0} \{\Delta_{j,2k}\}\times \overline{\Pp}_{j+2k+1}$$ as follows. Let $k\geq 0$.

\begin{enumerate}

\item For $(\Delta_{j,2k},\mu) \in \{\Delta_{j,2k}\}\times\left(\overline{\F}_{j+2k+1}\cap \overline{\Pp}_{j+2k+1}\right)$, do the transformation 
$$\Delta_{j,2k},\mu_1,\ldots,\mu_{d_{1+j+2k}^\mu},1+j+2k \mapsto \Delta_{j,2k+2},\mu_1+1,\ldots,\mu_{d_{1+j+2k}^\mu-1}+1\,.$$
This means that, in $\Delta_{j,2k}$, we add the parts $1+j+2k,j+2k+2$, and in $\mu$, we add $1$ to the $d_{1+j+2k}^\mu-1$ largest finite parts and delete the parts $\mu_{d_{1+j+2k}^\mu}=d_{1+j+2k}^\mu+1+j+2k$ and $1+j+2k$ to obtain a partition $\la$. We finally set $\psi_j((\Delta_{j,2k},\mu))=(\Delta_{j,2k+2},\la)$. Note that $|\la|=|\mu|-(j+2k+2)-(j+2k+1)$, and $\ell(\la)=\ell(\mu)-2$.

\item For $(\Delta_{j,2k},\mu) \in \{\Delta_{j,2k}\}\times\left(\F_{j+2k+1}\cap \overline{\Pp}_{j+2k+1}\right)$, do the transformation
$$\mu_1,\ldots,\mu_{d_{1+j+2k}^\mu},1+j+2k\mapsto\mu_1-1,\ldots,\mu_{d_{1+j+2k}^\mu}-1,d_{1+j+2k}^\mu+1+j+2k\,.$$
This means that, in $\mu$, we deleted the part $1+j+2k$, subtract $1$ to the $d_{1+j+2k}^\mu$ largest finite parts and add the part $d_{1+j+2k}^\mu+1+j+2k$ to obtain a partition $\la$. We finally set $\psi_j((\Delta_{j,2k},\mu))=(\Delta_{j,2k},\la)$.  Observe that $|\la|=|\mu|$ and $\ell(\la)=\ell(\mu)$.

\end{enumerate}

For all $(\Delta_{j,0},\mu)\in \{\Delta_{j,0}\}\times \F_j$, 
iterate the map $\psi_j$ as long  as it is possible. The iteration stops as soon as we reach a pair in $\bigsqcup_{k\geq 0} \{\Delta_{j,2k}\}\times \Pp_{j+2k+1}$. We then set $\Psi_j((\Delta_{j,0},\mu)) = \psi_j^{u}((\Delta_{j,0},\mu))\in \bigsqcup_{k\geq 0} \{\Delta_{j,2k}\}\times \Pp_{j+2k+1}$. Similarly to $\Phi_j$, if $\Psi_j$ is well-defined, it then induces a map from $\{\Delta_{j,0}\}\times \left(\F_j\cap \overline{\Pp}_j\right)$ to $\bigsqcup_{k\geq 0} \{\Delta_{j,2k}\}\times\left(\Pp_{j+2k+1}\cap\overline{\Pp}_j\right)$.

The reader can easily check that, by applying the corresponding $\Psi_j$ to the pairs obtained in Example \ref{ex:main}, we retrieve the pairs corresponding to $(11,8,7,7,5,5,4,3,2,2)$ by the exact inverse process.


\section{Proof of the well-definedness of the bijection}\label{sec:proofmain}

The maps $\Phi_j$ and $\Psi_j$ preserve the weight  and length sums of the pair of partitions, as they result from iterations of $\phi_j$ and $\psi_j$ which have this property. To prove that $\Phi_j$ and $\Psi_j$ describe inverse bijections, we first prove that $\phi_j$ and $\psi_j$ are inverse each other, then prove the well-definedness of $\Phi_j$ and $\Psi_j$ and  conclude.

\subsection{The maps $\phi_j$ and $\psi_j$ are inverse each other}

We first prove that $\phi_j$ and $\psi_j$ are inverse each other with the following result.

\begin{prop} \label{prop:phipsi}
For all $k\geq 0$, $\phi_j$ and $\psi_j$ describe inverse bijections between
 
\begin{enumerate}

\item $\{\Delta_{j,2k}\}\times \overline{\F}_{j+2k}$ and $\{\Delta_{j,2k}\}\times \left(\F_{1+j+2k}\cap\overline{\Pp}_{j+2k+1}\right)$,

\item $\{\Delta_{j,2k+2}\}\times \F_{j+2k+2}$ and $\{\Delta_{j,2k}\}\times \left(\overline{\F}_{1+j+2k}\cap\overline{\Pp}_{j+2k+1}\right)$.

\end{enumerate}

\end{prop}

\begin{proof}
Let $k,j\geq 0$.
\begin{enumerate}

\item The maps $\phi_j$ and $\psi_j$ describe inverse bijections between $\{\Delta_{j,2k}\}\times \overline{\F}_{j+2k}$ and\\
 $\{\Delta_{j,2k}\}\times \left(\F_{1+j+2k}\cap\overline{\Pp}_{j+2k+1}\right)$.
 
\begin{enumerate}

\item For $(\Delta_{j,2k},\la) \in \{\Delta_{j,2k}\}\times \overline{\F}_{j+2k}$, 
set $\ell(\la)\geq t\geq d_{j+2k}^\la$ such that $\la_t\geq 1+j+2k>\la_{t+1}$. By the first case of Section \ref{sec:Phi}, $\mu$ consists of the parts 
$$
\begin{cases}
\mu_i=\la_i+1 \ \ \text{ for } \ \ 1,\leq i < d_{j+2k}^\la\,,\\
\mu_i=\la_{i+1}\ \ \ \ \text{ for } \ \ d_{j+2k}^\la\leq i < t\,,\\
\mu_i=\la_i \ \ \ \ \ \ \ \text{ for } \ \ t< i \leq \ell(\la)\,,\\
\mu_t= 1+j+2k \,.
\end{cases}
$$
Since $1+j+2k$ is a part of $\mu$, we have that $\mu \in \overline{\Pp}_{1+j+2k}$.
Moreover, $\la_{d_{j+2k}^\la}\geq \mu_{d_{j+2k}^\la}$, and
\begin{align*}
\mu_{d_{j+2k}^\la-1}-(d_{j+2k}^\la-1)&=\la_{d_{j+2k}^\la-1}-(d_{j+2k}^\la-1)+1\\
&>\la_{d_{j+2k}^\la}-d_{j+2k}^\la+1\\
&= 1+j+2k\\
&>\mu_{d_{j+2k}^\la}-d_{j+2k}^\la\,,
\end{align*}
so that $d_{1+j+2k}^\mu = d_{j+2k}^\la-1$ and $\mu\in \F_{1+j+2k}$. Thus, 
$$\phi_j((\Delta_{j,2k},\la))\in \{\Delta_{j,2k}\}\times \left(\F_{1+j+2k}\cap\overline{\Pp}_{j+2k+1}\right)\,.$$
As $d_{1+j+2k}^\mu = d_{j+2k}^\la-1$, by applying the second case of Section \ref{sec:Psi} on $(\Delta_{j,2k},\mu)$, it is straightforward that $\psi_j(\phi_j((\Delta_{j,2k},\la)))=(\Delta_{j,2k},\la)$.

\item Let $(\Delta_{j,2k},\mu) \in \{\Delta_{j,2k}\}\times \left(\F_{1+j+2k}\cap\overline{\Pp}_{j+2k+1}\right)$.
Similarly to the previous case, as
$$\mu_{d_{1+j+2k}^\mu+1}\leq 1+d_{1+j+2k}^\mu+j+2k\leq \mu_{d_{1+j+2k}^\mu}-1\,,$$
we observe that the part $\la_{d_{1+j+2k}^\mu+1}=1+d_{1+j+2k}^\mu+j+2k$ so that $d_{j+2k}^\la = d_{1+j+2k}^\mu+1$ and $\la\in \F_{j+2k}$. Then, $$\psi_j((\Delta_{j,2k},\mu)) \in \{\Delta_{j,2k}\}\times \overline{\F}_{j+2k}$$. Using the first case of Section \ref{sec:Phi} on $(\Delta_{j,2k},\la)$ with $d_{j+2k}^\la = d_{1+j+2k}^\mu+1$, it is straightforward that $\phi_j(\psi_j((\Delta_{j,2k},\mu)))=(\Delta_{j,2k},\mu)$.
\end{enumerate}

 \item The maps $\phi_j$ and $\psi_j$ describe inverse bijections between $\{\Delta_{j,2k+2}\}\times \F_{j+2k+2}$ and\\
 $\{\Delta_{j,2k}\}\times \left(\overline{\F}_{1+j+2k}\cap\overline{\Pp}_{j+2k+1}\right)$.
 
 \begin{enumerate}
 
\item Let $(\Delta_{j,2k+2},\la) \in \{\Delta_{j,2k+2}\}\times \F_{j+2k+2}$. 
Note that $$\la_{d_{j+2k+2}^\la}-d_{j+2k+2}^\la>j+2k+2>\la_{d_{j+2k+2}^\la+1}-(d_{j+2k+2}^\la+1)\,.$$ 
Then, the partition $\mu$ consists of the parts $\la_1-1,\ldots,\la_{d_{j+2k+2}^\la}-1,\la_{d_{j+2k+2}^\la},\ldots,\la_{\ell}$ and the parts $d_{j+2k+2}^\la+j+2k+2,1+j+2k$. Since $1+j+2k$ is a part $\mu$, $\mu\in \overline{\Pp}_{1+j+2k}$. Moreover, $\la_{d_{j+2k+2}^\la}-1\geq d_{j+2k+2}^\la+j+2k+2\geq \la_{d_{j+2k+2}^\la}$ and $d_{j+2k+2}^\la+j+2k+2>1+j+2k$, so that $\mu_{d_{j+2k+2}^\la+1}=d_{j+2k+2}^\la+j+2k+2$. Hence, $d_{1+j+2k}^\mu = d_{j+2k+2}^\la+1$ and $\mu \in \overline{\F}_{1+j+2k}$, and 
$$\phi_j((\Delta_{j,2k+2},\la))\in \{\Delta_{j,2k}\}\times \left(\overline{\F}_{1+j+2k}\cap\overline{\Pp}_{j+2k+1}\right)\,.$$
Finally, by using the first case of Section \ref{sec:Psi} on $(\Delta_{j,2k},\mu)$ with $d_{1+j+2k}^\mu = d_{j+2k+2}^\la+1$, we retrieve the fact that $\psi_j(\phi_j((\Delta_{j,2k+2},\la)))=(\Delta_{j,2k+2},\la)$.

\item Let $(\Delta_{j,2k},\mu) \in \{\Delta_{j,2k}\}\times \left(\overline{\F}_{1+j+2k}\cap\overline{\Pp}_{j+2k+1}\right)$. 
Using the first case of Section \ref{sec:Psi},  we have that $\la$ consists of the parts $\mu_1+1,\ldots,\mu_{d_{1+j+2k}^\mu-1}+1$, and the parts $\mu_{d_{1+j+2k}^\mu+1},\ldots,\mu_{\ell(\mu)}$ except $1+j+2k$. Moreover, $\mu_{d_{1+j+2k}^\mu-1}+1-(d_{1+j+2k}^\mu-1)> \mu_{d_{1+j+2k}^\mu}-d_{1+j+2k}^\mu +1 = 2+j+2k > \mu_{d_{1+j+2k}^\mu+1}-d_{1+j+2k}^\mu$, so that $d_{j+2k+2}^\la = d_{1+j+2k}^\mu-1$ and $\la \in \F_{j+2k+2}$. Hence, 
$$\psi_j((\Delta_{j,2k},\mu)) \in \{\Delta_{j,2k+2}\}\times \F_{j+2k}\,.$$
We prove similarly to the previous case that $\phi_j(\psi_j((\Delta_{j,2k},\mu)))=(\Delta_{j,2k},\mu)$.

\end{enumerate}

\end{enumerate} 

\end{proof}

\subsection{Well-definedness of $\Phi_j$}

By Proposition \ref{prop:phipsi}, $\phi_j$ is injective, and this implies that a pair is not fixed by $\phi_j$ if and only if its iterations are not fixed. Hence, it suffices to check that  $(\Delta_{j,2k},\la) \in \{\Delta_{j,2k}\}\times \Pp_{1+j+2k}$ is not fixed by $\phi_j$, and that we reach $\{\Delta_{j,0}\}\times \F_j$ after a finite number of iterations of $\phi_j$. In this regard, we state the following proposition.

\begin{prop}\label{prop:phifinite}
Let $(\Delta_{j,2k},\la) \in \{\Delta_{j,2k}\}\times \overline{\F}_{j+2k}$. Then, $\la_{d_{j+2k}^\la}=d_{j+2k}^\la+j+2k$, and $\la_1\geq 1+j+2k$.

\begin{enumerate}

\item If $\la_1=1+j+2k$, then $d_{j+2k}^\la=1$ and $(\Delta_{j,2k},\la)$ is fixed by $\phi_j$. Inversely, all the pairs fixed by $\phi_j$ have the form $(\Delta_{j,2k},\la)$ with $\la_1=1+j+2k$.

\item If $\la_1>1+j+2k$, then $d_{j+2k}^\la\geq 2$ and, by setting $u=\sharp\{i\geq d_{j+2k}^\la:\la_i=d_{j+2k}^\la+j+2k\}$, we have that $$\phi_j^0((\Delta_{j,2k},\la)),\ldots,\phi_j^{u-1}((\Delta_{j,2k},\la))\in \{\Delta_{j,2k}\}\times \overline{\F}_{j+2k}\,,$$
and $\phi_j^{u}((\Delta_{j,2k},\la))\in \{\Delta_{j,2k}\}\times \F_{j+2k}$.

\end{enumerate}

\end{prop}

We first prove the well-definedness of $\Phi_j$ assuming that the above proposition is true. Let $\la \in \Pp_{1+j+2k}$. Since $1+j+2k$ is not  a part of $\la$, $\la_1\neq 1+j+2k$ and by Proposition \ref{prop:phifinite}, 
$(\Delta_{j,2k},\la)$ is not fixed by $\Phi_j$. Hence, its iterations  are not fixed by $\phi_j$. We can then use the second case of Proposition \ref{prop:phifinite} and deduce the existence of non-negative integers $u_l$ that counts the numbers of iterations of $(\Delta_{j,2k},\la)$ respectively in $\{\Delta_{j,2l}\}\times \overline{\F}_{j+2l}$. More precisely, we have for all $ l\in \{0,\ldots,k\}$,
\small{
$$\phi_j^{k-l+u+\sum_{t=l+1}^k u_t}((\Delta_{j,2k},\la))\in \{\Delta_{j,2l}\}\times \overline{\F}_{j+2l}\text{ for }0\leq u< u_l\text{ and }\phi_j^{k-l+\sum_{t=l}^k u_t}((\Delta_{j,2k},\la))\in \{\Delta_{j,2l}\}\times \F_{j+2l}\,.$$
}
Since these iterations are not fixed, by setting $n=|\la|+|\Delta_{j,2k}|$, we have by definition of $u_l$ in Proposition \ref{prop:phifinite} that 
$u_l+1\leq \frac{n}{2+j+2l}\leq \frac{n}{2}\cdot \frac{1}{l+1}$. Therefore, there is at most 
$\frac{n}{2}(1+\log(k+2))$ iterations, hence finite. Moreover, $$\Phi_j((\Delta_{j,2k},\la))=\phi_j^{k+\sum_{t=0}^k u_t}((\Delta_{j,2k},\la))\in \{\Delta_{j,0}\}\times \F_{j}$$ so that $\Phi_j((\Delta_{j,2k},\la))$ is well-defined.

\begin{proof}[Proof of Proposition \ref{prop:phifinite}]
Let $(\Delta_{j,2k},\la) \in \{\Delta_{j,2k}\}\times \overline{\F}_{j+2k}$.
Recall that $\la_{d_{j+2k}^\la}=d_{j+2k}^\la+j+2k$ and set $\phi_j((\Delta_{j,2k},\la))=(\Delta_{j,2k},\mu)$.

\begin{enumerate}

\item As $\la_1\geq 1+j+2k$, the fact that $d_{j+2k}^\la$ is unique implies that $\la_1=1+j+2k$ if and only if $d_{j+2k}^\la=1$. The fact that $\phi_j(\Delta_{j,2k},\la)=(\Delta_{j,2k},\la)$ is trivial, as $\mu$ is obtained by deleting $\la_1$, adding $1+j+2k$, and not adding $1$ to any other part. Inversely, only the pairs of $\{\Delta_{j,2k}\}\times \overline{\F}_{j+2k}$ can be fixed by $\phi_j$, and when $d_{j+2k}^\la\geq 2$, $\mu_1=\la_1+1>\la_1$ so that $\phi_j(\Delta_{j,2k},\la)\neq (\Delta_{j,2k},\la)$. Hence, the only pairs fixed by $\phi_j$ have the form $(\Delta_{j,2k},\la)$ with $\la_1=1+j+2k$.

\item When $\la_1>1+j+2k$, $d_{j+2k}^\la\geq 2$. Since $\mu_{d_{j+2k}^\la-1}>\la_{d_{j+2k}^\la}\geq \mu_{d_{j+2k}^\la}$, we then have that $\mu_{d_{j+2k}^\la-1}-(d_{j+2k}^\la-1)>j+2k\geq \mu_{d_{j+2k}^\la}-{d_{j+2k}^\la}$. Hence, $\mu \in \overline{\F}_{j+2k}$ if and only if $d_{j+2k}^\la=d_{j+2k}^\mu$  and $\mu_{d_{j+2k}^\la}=d_{j+2k}^\la+j+2k>1+j+2k$. This  occurs only if $\mu_{d_{j+2k}^\la}=\la_{d_{j+2k}^\la+1}=d_{j+2k}^\la+j+2k$.
Thus, by formally setting $(\Delta_{j,2k},\la^{(w)})=\phi_j^w((\Delta_{j,2k},\la))$ for $w\in \{0,\ldots,u\}$ and $\ell(\la)\geq t\geq d_{j+2k}^\la+u-1$ such that $\la_t\geq 1+j+2k>\la_{t+1}$, we have that
$$
\begin{cases}
\la_i^{(w)}=\la_i+w \ \ \ \ \ \ \ \ \text{ if } \ \ 1\leq i\leq d_{j+2k}^\la-1\,,\\
\la_i^{(w)}=\la_{i+w} \ \ \ \ \ \ \ \ \ \ \text{ if } \ \ d_{j+2k}^\la\leq i\leq t-w\,,\\
\la_i^{(w)}=1+j+2k \ \ \ \text{ if } \ \ t-w+1\leq i\leq t\,,\\
\la_i^{(w)}=\la_{i} \ \ \ \ \ \ \ \ \ \ \ \ \ \ \text{ if } \ \ t+1\leq i\leq \ell(\la)\,,
\end{cases}
$$ 
as we recursively obtain that $d_{j+2k}^{\la^{(w)}}=d_{j+2k}^\la$ and $\la_{d_{j+2k}^\la}^{(w)}=\la_{d_{j+2k}^\la+w}=d_{j+2k}^\la+j+2k$ for all $w\in \{0,\ldots,u-1\}$. Moreover, $\la_{d_{j+2k}^\la}^{(u)}$ is either $\la_{d_{j+2k}^\la+u}$ when $t\geq d_{j+2k}^\la+u$, or $1+j+2k$ when $t= d_{j+2k}^\la+u-1$. In all cases, $\la_{d_{j+2k}^\la}^{(u)}-d_{j+2k}^\la<j+2k<\la_{d_{j+2k}^\la-1}^{(u)}-(d_{j+2k}^\la-1)-u$, so that $\la^{(u)}\in \F_{j+2k}$.

\end{enumerate}

\end{proof}

\subsection{Well-defined of $\Psi_j$}

As before, we only need to check that $(\Delta_{j,0},\la)\in \{\Delta_{j,0}\}\times \F_j$ is not fixed by $\psi_j$, and that we reach $\bigsqcup_{k\geq 0} \{\Delta_{j,2k}\}\times \Pp_{j+2k+1}$ after a finite number of iterations of $\psi_j$. The following propositions helps in that purpose.

\begin{prop}\label{prop:psifinite}
For $(\Delta_{j,2k},\mu) \in \{\Delta_{j,2k}\}\times \left(\F_{1+j+2k}\cap\overline{\Pp}_{j+2k+1}\right)$, we have $\mu_{d_{1+j+2k}^\mu}-d_{1+j+2k}^\mu>1+j+2k>\mu_{d_{1+j+2k}^\mu+1}-(d_{1+j+2k}^\mu+1)$.

\begin{enumerate}

\item If $d_{1+j+2k}^\mu=0$, then $\mu_1=1+j+2k$ and $(\Delta_{j,2k},\mu)$ is fixed by $\psi_j$. Inversely, all the pairs fixed by $\psi_j$ have the form $(\Delta_{j,2k},\mu)$ with $\mu_1=1+j+2k$.

\item If $d_{1+j+2k}^\mu\geq 1$, set $u=\mu_{d_{1+j+2k}^\mu}-d_{1+j+2k}^\mu-(1+j+2k)$, and $v$ the number occurrences of  $1+j+2k$ in $\mu$.

\begin{enumerate}

\item If $v>u$, then 
$$\psi_j^0((\Delta_{j,2k},\mu)),\ldots,\psi_j^{u-1}((\Delta_{j,2k},\mu))\in \{\Delta_{j,2k}\}\times \left(\F_{1+j+2k}\cap\overline{\Pp}_{j+2k+1}\right)\,,$$
and $\psi_j^{u}((\Delta_{j,2k},\mu))\in \{\Delta_{j,2k}\}\times \left(\overline{\F}_{1+j+2k}\cap\overline{\Pp}_{j+2k+1}\right)$.

\item If $v\leq u$, then 
$$\psi_j^0((\Delta_{j,2k},\mu)),\ldots,\psi_j^{v-1}((\Delta_{j,2k},\mu))\in \{\Delta_{j,2k}\}\times \left(\F_{1+j+2k}\cap\overline{\Pp}_{j+2k+1}\right)\,,$$
and $\psi_j^{v}((\Delta_{j,2k},\mu))\in \{\Delta_{j,2k}\}\times \Pp_{j+2k+1}$.

\end{enumerate} 

\end{enumerate}

\end{prop}

To prove the well-definedness of $\Psi_j$, we first observe that, for $\mu \in \F_j$, $\mu_1\neq j+1$, and by Proposition \ref{prop:psifinite}, $(\Delta_{j,0},\mu)$ is not fixed by $\psi_j$. Thus, as $\psi_j$ is injective, its iterations are not fixed by $\psi_j$. Moreover, for $l$ such that $|\Delta_{j,2l}|=l(1+2j+2l))>|\mu|$, these iterations do not reach $\{\Delta_{j,2l}\}\times \Pp$ since $\psi_j$ conserves the weight sum of the pair. Let then $k$ be the largest $l$ such that $\{\Delta_{j,2l}\}\times \Pp$ is reached by the iterations. By Proposition \ref{prop:psifinite}, for $0\leq l\leq k$, $w_l$ the number of iterations of $(\Delta_{j,0},\mu)$ by $\psi_j$ in $\{\Delta_{j,2l}\}\times \Pp$ is at most equal to one plus the number of occurrences of $1+j+2l$, hence at most equal to $\frac{|\mu|}{1+j+2l}$. We then have in total at most $\frac{|\mu|(2+\log(k+1))}{2}$ iterations. Finally, at $k$, 
$$\Psi_j((\Delta_{j,0},\mu)) = \psi_j^{k+\sum_{l=0}^k w_l}((\Delta_{j,0},\mu))$$
is necessarily in $\{\Delta_{j,2k}\}\times \Pp_{j+2k+1}$ as the iterations stop. Hence, $\Psi_j$ is well-defined.

\begin{proof}[Proof of Proposition \ref{prop:psifinite}]
As $1+j+2k$ is a part of $\mu$, we then have $\mu_1\geq 1+j+2k$. Set $\psi_j((\Delta_{j,2k},\mu))=(\Delta_{j,2k},\la)$.

\begin{enumerate}

\item If $d_{1+j+2k}^\mu=0$, then $\mu_1\leq 1+j+2k$, so that $\mu_1=1+j+2k$. One can easily check that such pair $(\Delta_{j,2k},\mu)$ is fixed by $\psi_j$. Inversely, the only pairs fixed by $\psi_j$ are in $\{\Delta_{j,2k}\}\times \left(\F_{1+j+2k}\cap\overline{\Pp}_{j+2k+1}\right)$ for some $k\geq 0$. By Proposition \ref{prop:phipsi}, the pairs of $\{\Delta_{j,2k}\}\times \left(\F_{1+j+2k}\cap\overline{\Pp}_{j+2k+1}\right)$ fixed by $\psi_j$ are exactly the pairs of $\{\Delta_{j,2k}\}\times \F_{j+2k}$ fixed by $\phi_j$. Finally, Proposition \ref{prop:phifinite} gives us the form of the fixed pairs, which is $(\Delta_{j,2k},\mu)$ with $\mu_1=1+j+2k$.

\item If $d_{1+j+2k}^\mu\geq 1$, then $\mu_1\geq \mu_{d_{1+j+2k}^\mu}>1+j+2k+d_{1+j+2k}^\mu$. By the second part of Section \ref{sec:Psi},
$$\mu_{d_{1+j+2k}^\mu}-1\geq 1+j+2k+d_{1+j+2k}^\mu \geq \mu_{d_{1+j+2k}^\mu+1}\,,$$ so that $\la_{d_{1+j+2k}^\mu+1}=1+j+2k+d_{1+j+2k}^\mu$. Hence,
$$\la_{d_{1+j+2k}^\mu}-d_{1+j+2k}^\mu\geq 1+j+2k>\la_{d_{1+j+2k}^\mu+1}-(d_{1+j+2k}^\mu+1)\,,$$ and $d_{1+j+2k}^\la=d_{1+j+2k}^\mu$. Therefore,
both $u=\mu_{d_{1+j+2k}^\mu}-(1+j+2k+d_{1+j+2k}^\mu)$ and $v$ the number of occurrences of $1+j+2k$ decrease by $1$ after $\psi_j$ is applied once, while $d_{1+j+2k}^\mu$ is conserved. Let $\ell(m)\geq t\geq d_{1+j+2k}^\mu+v$ such that $\mu_t=1+j+2k>\mu_{t+1}$. 

\begin{enumerate}

\item If $v>u$, by formally setting $(\Delta_{j,2k},\mu^{(w)})=\psi_j^w((\Delta_{j,2k},\mu))$ for $w\in \{0,\ldots,u\}$, we have that
$$
\begin{cases}
\mu_i^{(w)}=\mu_i-w \ \ \ \ \ \ \ \ \ \ \ \ \ \ \ \ \ \ \ \ \ \ \text{ if } \ \ 1\leq i\leq d_{1+j+2k}^\mu\,,\\
\mu_i^{(w)}=1+d_{1+j+2k}^\mu+j+2k \ \ \ \text{ if } \ \ d_{1+j+2k}^\mu+1\leq i\leq d_{1+j+2k}^\mu+w\,,\\
\mu_i^{(w)}=\mu_{i-w} \ \ \ \ \ \ \ \ \ \ \ \ \ \ \ \ \ \ \ \ \ \ \ \ \text{ if } \ \ d_{1+j+2k}^\mu+w+1\leq i\leq t\,,\\
\mu_i^{(w)}=\mu_{i} \ \ \ \ \ \ \ \ \ \ \ \ \ \ \ \ \ \ \ \ \ \ \ \ \ \ \ \ \text{ if } \ \ t+1\leq i\leq \ell(\mu)\,,
\end{cases}
$$ 
as we recursively obtain that $d_{1+j+2k}^{\mu^{(w)}}=d_{1+j+2k}^\mu$,
$$\mu_{d_{1+j+2k}^\mu}^{(w)}=\mu_{d_{1+j+2k}^\mu}-w>d_{1+j+2k}^\mu+1+j+2k$$ and $\mu_{t}^{(w)}=1+j+2k$ for all $w\in \{0,\ldots,u-1\}$.  Finally, we have $\mu_{d_{1+j+2k}^\mu}^{(u)}=d_{1+j+2k}^\mu+1+j+2k$ and $\mu_{t}^{(u)}=1+j+2k$ so that $\mu^{(u)}\in \overline{\F}_{1+j+2k}\cap \overline{\Pp}_{1+j+2k}$.

\item If $v\leq u$, by formally setting $(\Delta_{j,2k},\mu^{(w)})=\psi_j^w((\Delta_{j,2k},\mu))$ for $w\in \{0,\ldots,v\}$, we have that
$$
\begin{cases}
\mu_i^{(w)}=\mu_i-w \ \ \ \ \ \ \ \ \ \ \ \ \ \ \ \ \ \ \ \ \ \ \text{ if } \ \ 1\leq i\leq d_{1+j+2k}^\mu\,,\\
\mu_i^{(w)}=1+d_{1+j+2k}^\mu+j+2k \ \ \ \text{ if } \ \ d_{1+j+2k}^\mu+1\leq i\leq d_{1+j+2k}^\mu+w\,,\\
\mu_i^{(w)}=\mu_{i-w} \ \ \ \ \ \ \ \ \ \ \ \ \ \ \ \ \ \ \ \ \ \ \ \ \text{ if } \ \ d_{1+j+2k}^\mu+w+1\leq i\leq t\,,\\
\mu_i^{(w)}=\mu_{i} \ \ \ \ \ \ \ \ \ \ \ \ \ \ \ \ \ \ \ \ \ \ \ \ \ \ \ \ \text{ if } \ \ t+1\leq i\leq \ell(\mu)\,,
\end{cases}
$$ 
as we recursively obtain that $d_{1+j+2k}^{\mu^{(w)}}=d_{1+j+2k}^\mu$,  $$\mu_{d_{1+j+2k}^\mu}^{(w)}=\mu_{d_{1+j+2k}^\mu}-w>d_{1+j+2k}^\mu+1+j+2k$$
and $\mu_t=1+j+2k$ for all $w\in \{0,\ldots,v-1\}$. Finally, we have $d_{1+j+2k}^{\mu^{(w)}}=u-v+d_{1+j+2k}^\mu+1+j+2k$, and $\mu_{t}^{(u)}$ is either $1+d_{1+j+2k}^\mu+j+2k$ when $t=d_{1+j+2k}^\mu+v$, or $\mu_{t-v}>1+j+2k$ when $t>d_{1+j+2k}^\mu+v$, so that $\mu^{(u)}\in \overline{\Pp}_{1+j+2k}$.

\end{enumerate}

\end{enumerate}

\end{proof}

\subsection{The maps $\Phi_j$ and $\Psi_j$ are inverse each other}

For $(\Delta_{j,2k},\la)\in \{\Delta_{j,2k}\}\times \Pp_{1+j+2k}$, there exists a unique finite non-negative integer $u$ such that $\Phi_j((\Delta_{j,2k},\la))=\phi_j^u((\Delta_{j,2k},\la))\in \{\Delta_{j,2k}\}\times \F_j$. Then, by Proposition \ref{prop:phipsi}, $\psi_j^u(\Psi_j((\Delta_{j,2k},\la)))=(\Delta_{j,2k},\la)$, and as it belongs to $\{\Delta_{j,2k}\}\times \Pp_{1+j+2k}$, it is by definition $\Psi_j(\Phi_j((\Delta_{j,2k},\la)))$. Similarly, we prove that,
for $\mu \in \F_j$, $\Phi_j(\Psi_j((\Delta_{j,0},\mu)))=(\Delta_{j,0},\mu)$. 

Finally, since the bijections preserve the part less or equal to $j$, $\Phi_j$ then induces a bijection from $\M_j\cap\overline{\Pp}_j$ to $\F_j\cap\overline{\Pp}_j$ and $\Psi_j=\Phi_j^{-1}$.

\section{Bijection for Corollary \ref{cor:mainbis}}\label{sec:corbijection}

We here provide a bijection of Corollary \ref{cor:mainbis} in the spirit of Section \ref{sec:mainbijection}. First, observe the following correspondence. 

\begin{lem}
There is a weight-preserving bijection between $\overline{\F}_j$ and $\{\Delta_{j,1}\}\times \F_{j+1}$.
\end{lem}

\begin{proof}
For all $\la$ in $\overline{\F}_j$, we have $\la_{d_j^\la} = j+d_j^\la$. Hence, set $\psi'_j(\la)$ to be the pair $((j+1),\mu)$, where $\mu$ consists of the parts $\la_1+1,\ldots,\la_{d_j^\la-1}+1$ and $\la_i$ for $i>d_j^\la$. Inversely, for $((j+1),\mu)\in \{\Delta_{j,1}\}\times \F_{j+1}$, we set $\phi'_j((j+1),\mu))=\la$ whose parts are $\mu_1-1,\ldots,\mu_{d_{j+1}^\mu}-1$, $j+1+d_{j+1}^\mu $and $\mu_i$ for $i>d_{j+1}^\mu$. The proof that $\phi'_j$ and $\psi'j$ are inverse each other is similar to the proof of Proposition \ref{prop:phipsi}, as $d_j^\la =d_{j+1}^\mu+1$.
\end{proof}

To bijectively prove Corollary \ref{cor:mainbis}, 
we build two maps,
 $$\Phi'_j  : \bigsqcup_{k\geq 0} \{\Delta_{j,2k+1}\}\times \Pp_{j+2k+2}\to \{\Delta_{j,1}\}\times \F_{j+1}$$
 and 
 $$\Psi'_j : \{\Delta_{j,1}\}\times \F_{j+1} \to \bigsqcup_{k\geq 0} \{\Delta_{j,2k+1}\}\times \Pp_{j+2k+2}.$$
The map $\Phi'_j$ is simply obtained by going through the process of $\Phi_j$, except that we replace all the occurrences of ``$2k$'' by ``$2k+1$''. Similarly, $\Phi'_j$ is obtained by $\Psi_j$ by replacing ``$2k$'' by ``$2k+1$''. The proof of the well-definedness of the bijection is the same as the proof provided in Section \ref{sec:proofmain}.

\section{Bijective proof of Theorem \ref{theo:crank-frobenius}}\label{sec:crank-frobenius}

Before constructing the bijection for Theorem \ref{theo:crank-frobenius}, we first state the key result given by Hopkins, Sellers and Yee in \cite{HSY}, and that provides a combinatorial link between the crank and the Durfee decomposition. 

\begin{lem}[Hopkins-Sellers-Yee]\label{lem:crank-frob}
Let $j\in \Zo$ and $\la\in \Pp$. Then,
$$crank(\la)\leq -j \text{ if and only if } \omega(\la)\geq d_j^\la+j\,\cdot$$
\end{lem}

\begin{rem}
For $j\in \Zo$ and $\la\in \Pp$, Lemma \ref{lem:crank-frob}  implies that $crank(\la)= -j$ if and only if $d_{j+1}^\la+j+1>\omega(\la)\geq  d_j^\la+j$. By Remark \ref{rem:dj}, it equivalently means that $\omega(\la)=d_j^\la+j, \eta(\la)=d_j^\la$ and $d_{j+1}^\la=d_j^\la$.
\end{rem}

\begin{proof}[Proof of Lemma \ref{lem:crank-frob}]
We have that $crank(\emptyset)=0$ and $d_j^\emptyset = 0$. The equivalence then stands for $\la=\emptyset$. Now suppose that $\la\neq \emptyset$, which equivalently means that $\la_1>0$ and $d_0^\la>0$.

\begin{enumerate}
\item If $\omega(\la)\geq d_j^\la+j$, then, by Remark \ref{rem:dj}, $\omega(\la)\geq d_0^\la>0$, and $\eta(\la)\leq d_j^\la$ as $\la_{d_j^\la+1}\leq d_j^\la+j \leq \omega(\la)$. Therefore, $crank(\la)=\eta(\la)-\omega(\la)\leq d_j^\la-(d_j^\la+j)=-j$.

\item Otherwise, if $0<\omega(\la)< d_j^\la+j$, then $\la_{d_j^\la}>\omega(\la)$ so that $\eta(\la)\geq d_j^\la$. Hence, $crank(\la)=\eta(\la)-\omega(\la)> d_j^\la-(d_j^\la+j)=-j$. Finally, if $\omega(\la)=0$, $\la_1>0$ so that $crank(\la)>0\geq -j$.

\end{enumerate}

\end{proof}

For an integer $n$, set $\C_{\leq n}=\{\la\in\Pp:crank(\la)\leq n\}$ and $\C_{\geq n}=\{\la\in\Pp:crank(\la)\geq n\}$. Theorem \ref{theo:crank-frobenius} is then equivalent to saying that, for $j\in \Zo$, there exists a weight-preserving bijection $\Gamma_j$ between $\F_j\cap \overline{\Pp}_j \setminus\{\emptyset,(1)\}$ and $\C_{\leq -j}\setminus\{\emptyset,(1)\}$. Also, for $\la\neq \emptyset$, we then always have $d_j^\la+j>0$. The construction of $\Gamma_j$ is the following.

\begin{enumerate}

\item For $\la\in \F_j\cap \overline{\Pp}_j$ with $|\la|>1$, recall that $\la_{d_j^\la}-d_j^\la> j >\la_{d_j^\la+1}-(d_j^\la+1)$. Hence $\la_{d_j^\la}>j$. Let $\ell(\la)\geq t^\la\geq d_j$ such that $\la_{t^\la}>j=\la_{t^\la+1}$. The map $\Gamma_j$ consists in sustracting $1$ to the $d_j^\la$ largest finite parts, deleting the part $\la_{t^\la+1}=j$ and adding  $d_j^\la+j$ parts equal to $1$, so that it is obviously weight preserving. Formally, $\Gamma_j(\la)=\mu$ with
$$\begin{cases}
\mu_{i}=\la_i-1 \ \ \text{if} \ \ 1\leq i\leq d_j^\la\,,\\
\mu_{i}=\la_i\ \ \ \ \ \ \ \text{if} \ \ d_j^\la<i\leq t^\la\,,\\
\mu_{i}=\la_{i+1}\ \ \ \ \text{if} \ \ t^\la< i\leq \ell(\la)-\chi(j\geq 1)\,,\\
\mu_{i}=1 \ \ \ \ \ \ \ \ \text{if} \ \  \ell(\la)+\chi(j=0)\leq i\leq \ell(\la)+d_j^\la+j-\chi(j\geq 1) \,.
\end{cases}$$
Here $\chi(A)$ equals $1$ if $A$ is true and $0$ if not.
Observe that $\mu_{d_j^\la}-d_j^\la \geq j>\mu_{d_j^\la+1}-(d_j^\la+1)$, as $\mu_{d_j^\la+1}=\la_{d_j^\la+1}$ if $d_j^\la<t^\la$, or $\la_{d_j^\la+2}$ if $d_j^\la=t^\la<\ell(\la)-\chi(j\geq 1)$, or $1$ if $d_j^\la=t^\la=\ell(\la)-\chi(j\geq 1)$. Therefore, $d_j^\mu=d_j^\la$ and $\mu\in \C_{\leq -j}$. Finally, note that $\ell(\mu)=\ell(\la)+d_j^\la+j-\chi(j\geq 1)\geq 2d_j^\mu+j-\chi(j\geq 1)$.

\item Inversely, let $\emptyset\neq \mu\in \C_{-j}$ with $|\la|>1$. If $\mu_{d_j^\mu}=1$, then $d_j^\mu=1$ and $j=0$. In that case, as $\mu\neq (1)$, $\ell(\mu)\geq 2$, so that $\ell(\mu)\geq 2d_j^\mu$. If $\mu_{d_j^\mu}>1=\mu_{\ell(\mu)-d_j^\mu+j+1}$, then $\ell(\mu)\geq 2d_j^\mu+j$. We thus always have $\ell(\mu)\geq 2d_j^\mu+j$. For $j\geq 1$, let $d_j^\mu\leq t^\mu\leq \ell(\mu)-d_j^\mu-j$ such that $\mu_t>j\geq \mu_{t+1}$, and for $j=0$, set $t^\mu=\ell(\mu)-d_0^\mu$. The map $\Gamma^{-1}_j$ then consists in deleting the $d_j^\mu+j$ smallest parts equal to $1$, adding $1$ to the $d_j^\mu$ largest finite parts and a part $j$.
Formally, $\Gamma^{-1}_j(\mu)=\la$ with
$$\begin{cases}
\la_{i}=\mu_i+1 \ \ \text{if} \ \ 1\leq i\leq d_j^\mu\,,\\
\la_{i}=\mu_i \ \ \ \ \ \ \ \text{if} \ \ d_j^\mu< i\leq t^\mu\,,\\
\la_{i}=\mu_{i-1} \ \ \ \ \text{if} \ \ t^\mu+1<i\leq \ell(\mu)-d_j^\mu-j+1\,,\\
\la_{t^\mu+1}=j \,.
\end{cases}$$
As $\mu_{d_j^\mu}-d_j^\mu\geq j >\mu_{d_j^\mu+1}-(d_j^\mu+1)$ and $\mu_{d_j^\mu+1}\geq \la_{d_j^\mu+1}$, we have that $\la_{d_j^\mu}-d_j^\mu>j>\la_{d_j^\mu+1}-(d_j^\mu+1)$. Therefore, $d_j^\la=d_j^\mu$ and $\la\in \F_j\cap\overline{\Pp}_j$. Note that $\ell(\la)=\ell(\mu)-d_j^\mu-j+\chi(j\geq 1)$.

\end{enumerate}

The map $\Gamma_j$ is well-defined as $\Gamma_j((\F_j\cap\overline{\Pp}_j)\setminus\{\emptyset,(1)\})\subset \C_{\leq -j}$, and since $\mu=\Gamma_j(\la)$ satisfies $d_j^\mu=d_j^\la$, it is straightforward that $\Gamma^{-1}_j(\Gamma_j(\la))=\la$ as the corresponding $t^\mu$ equals $t^\la$. Inversely, $\Gamma^{-1}_j(\mu)\subset \F_j\cap\overline{\Pp}_j$ and $\Gamma_j(\Gamma^{-1}_j(\mu))=\mu$ for all $\mu\in \C_{\leq -j}$ with $|\la|>1$.

\begin{ex}\label{ex:crank-frobenius}
We have the following table:
\small{
$$
\begin{array}{|c|c|c|c|c|}
\hline
j&\la&d_j^\la&\Gamma_j(\la)&crank(\Gamma_j(\la))\\
\hline\hline
1&(1)&0&(1)&-1\\
0&\Delta_{0,2k}&k&(\underbrace{2k-1,\ldots,k}_{k\text{ consecutive}},\underbrace{k,\ldots,1}_{k\text{ consecutive}},\underbrace{1,\ldots,1}_{k})&-3\chi(k\geq 1)\\
j\geq 1&\Delta_{j-1,2k}&k&(\underbrace{2k-1+j,\ldots,k+j}_{k\text{ consecutive}},\underbrace{k+j,\ldots,1+j}_{k\text{ consecutive}},\underbrace{1,\ldots,1}_{k+j})&-j-\chi(k\geq 1)\\
0&(13,10,9,9,4,3,2,2,1,1)&4&(12,9,8,8,4,3,2,2,1,1,1,1,1,1)&-2\\
3&(12,9,9,5,4,4,4,3,2,2)&3&(11,8,8,5,4,4,4,2,2,1,1,1,1,1,1)&-3\\
5&(11,8,7,7,5,5,4,3,2,2)&2&(10,7,7,7,5,4,3,2,2,1,1,1,1,1,1,1)&-6\\
\hline
\end{array}\,.
$$
}
\end{ex}

\section{Bijective proof of Corollary \ref{cor:cranksym}}\label{sec:cor}

We here present a crank-sign reversing involution provided by Berkovich and Garvan in \cite{BG02}. The involution $\Lambda$ on $\Pp\setminus\{\emptyset,(1)\}$ is such that $crank(\Lambda(\la))=-crank(\la)$.  Let $\la\in \Pp$ with $|\la|>1$, and construct $\nu=\Lambda(\la)$ as follows.

\begin{enumerate}

\item 

\begin{enumerate}

\item If $\omega(\la)=0$, then $\la_1\geq \la_{\ell(\la)}>1$. Set
$$\begin{cases}
\nu_i=\la_{i+1} \ \ \text{if} \ \ 1\leq i\leq \ell(\la)-1\,,\\
\nu_i=1 \ \ \ \ \ \ \text{if} \ \ \ell(\la)\leq i\leq \la_1+\ell(\la)-1\,.
\end{cases}$$
Hence, $|\nu|=|\la|$, $\omega(\nu)=\la_1>0$ and $\eta(\nu)=0$ so that $crank(\nu)=-\la_1=-crank(\la)$.

\item If $\omega(\la)>0$ and $\eta(\la)=0$, then set
$$\begin{cases}
\nu_1=\omega(\la)\,\\
\nu_i= \la_{i-1}\ \ \ \ \ \ \text{if} \ \ 2\leq i\leq \ell(\la)-\omega(\la)+1\,.
\end{cases}$$
Hence, $|\nu|=|\la|$, $\omega(\nu)=0$, and $crank(\nu)=\omega(\la)=-crank(\la)$.

\end{enumerate}

One can easily check that these two cases are inverse each other.

\item If $\omega(\la),\eta(\la) >0$, let $\rho(\la)=\max\{\omega(\la),\la_2-1\}$ and let $\la^*$ be the conjugate of $\la$, which is defined by $\la^*_i=\sharp\{u: \la_u\geq i\}$ for all $i\in \{1,\ldots,\la_1\}$. Then, $\ell(\la^*)=\la_1$, $\eta(\la)=\la_{\omega(\la)+1}^*$ and $\omega(\la)=\la^*_1-\la_2^*$. We thus set 
$$\begin{cases}
\nu_1=\la_2^*+\la_1-\rho(\la)\,,\\
\nu_i= 1+\la_i^*\ \ \ \ \ \ \ \ \ \ \ \ \ \ \ \ \ \ \ \ \text{if} \ \ 2\leq i\leq\omega(\la)\,,\\
\nu_i=\la_{i+1}^*\ \ \ \ \ \ \ \ \ \ \ \ \ \ \ \ \ \ \ \ \ \ \text{if} \ \ \omega(\la)< i\leq \rho(\la)\,,\\
\nu_{i}=1\ \ \ \ \ \ \ \ \ \ \ \ \ \ \ \ \ \ \ \ \ \ \ \ \ \ \text{if} \ \ \rho(\la)<i\leq \rho(\la) +\eta(\la)\,.\\
\end{cases}$$ 
For all $\omega(\la)<i\leq \rho(\la)$, $2=\la^*_{\la_2}\leq \nu_i\leq \la^*_{\omega(\la)+2}\leq \eta(\la)$. Moreover, $\la^*_{2}\geq \la^*_{\omega(\la)+1}=\eta(\la)$, $\la_1-\rho(\la)\geq 1$ as $\la_1-\omega(\la)\geq 1$ and $\la_1-\la_2+1\geq 1$, and for all $2\leq i\leq \omega(\la)$, $\nu_i\geq 1+\la^*_{\omega(\la)}\geq 1+\eta(\la)$. Therefore, $\omega(\nu)=\eta(\la)$ and $\eta(\nu)=\omega(\la)$, and $crank(\nu)=-crank(\mu)$. Furthermore, since $\rho(\la)+1\geq \la_2$, $\la^*_i=1$ for all $\rho(\la)+1<i\leq \la_1$, and 
\begin{align*}
|\nu|&= \underbrace{\omega(\la)+\la_2^*}_{\la^*_1} + (\la_1-\rho(\la)-1)+\sum_{i=2}^{\omega(\la)} \la^*_i + \sum_{\omega(\la)+2}^{\rho(\la)+1} \la^*_i + \underbrace{\eta(\la)}_{\la^*_{\omega(\la)+1}}\\
&=  \sum_{i=1}^{\la_1} \la^*_i =|\la^*|
\end{align*}
so that $|\nu|=|\la|$. In addition, 
$$\begin{cases}
\nu_i^*= \la_i-1\ \ \ \ \ \ \ \ \ \ \ \ \ \ \ \ \ \ \ \ \text{if} \ \ 2\leq i\leq\eta(\la)\,,\\
\nu^*_i=\la_{i-1}\ \ \ \ \ \ \ \ \ \ \ \ \ \ \ \ \ \ \ \ \ \ \text{if} \ \ \eta(\la)+1< i\leq \la^*_2+1\,,\\
\nu^*_{\eta(\la)+1}=\omega(\la)\,,\\
\end{cases}$$
and $\rho(\nu)=\la^*_2$ if $\omega(\la)>1$ and $\rho(\nu)=\eta(\la)=\la^*_2$ if $\omega(\la)=1$ so that $\rho(\nu)=\la^*_2$. Finally, $\nu^*_2=\rho(\la)$ if $\rho(\la)>\omega(\la)$, and $\nu^*_2=\omega(\la)=\rho(\la)$ if $\rho(\la)>\omega(\la)$ as $\nu_{\omega(\la)}\geq \omega(\la)+1\geq 2$.  Hence, $\nu^*_2=\rho(\la)$, and for $\Lambda(\nu)=\kappa$, we have
$$\begin{cases}
\kappa_1= \rho(\la)+\nu_1-\la^*_2=\la_1\,,\\
\kappa_i= 1+\nu_i^* = \la_i\ \ \ \ \ \ \ \ \ \ \ \ \ \ \ \text{if} \ \ 2\leq i\leq\eta(\la)\,,\\
\kappa_i=\nu_{i+1}^*=\la_i\ \ \ \ \ \ \ \ \ \ \ \ \ \ \ \ \ \ \text{if} \ \ \eta(\la)< i\leq \la^*_2\,,\\
\kappa_{i}=1\ \ \ \ \ \ \ \ \ \ \ \ \ \ \ \ \ \ \ \ \ \ \ \ \ \ \ \ \ \text{if} \ \ \la^*_2<i\leq \la^*_2 +\omega(\la)=\la^*_1\,.\\
\end{cases}$$
We then conclude that $\Lambda(\Lambda(\la))=\la$.
\end{enumerate}

\begin{ex}\label{ex:cor}
We have the following table:
$$
\begin{array}{|c|c|c|c|c|}
\hline
\la&\omega(\la)&\eta(\la)&\rho(\la)&\Lambda(\la)\\
\hline\hline
(12,9,8,8,4,3,2,2,1,1,1,1,1,1)&6&4&8&(12,9,7,6,5,5,4,2,1,1,1,1)\\
(11,8,8,5,4,4,4,2,2,1,1,1,1,1,1)&6&3&7&(13,10,8,8,5,4,3,1,1,1)\\
(10,7,7,7,5,4,3,2,2,1,1,1,1,1,1,1)&7&1&7&(12,10,8,7,6,5,5,1)\\
\hline
\end{array}\,.
$$
\end{ex}


\section{Conclusion}\label{sec:conclusion}

As we construct the different intermediate bijections in Section \ref{sec:mainbijection}, \ref{sec:crank-frobenius} and \ref{sec:cor}, we now present the full scope of the bijection for Theorem \ref{theo:main}. For $j\in\Zo$, the bijection between $\M_j\cap \overline{\Pp}_j$ and the set of partitions with crank at least equal to $j$ is given by $\Lambda\circ \Gamma_j\circ \Phi_j$, and its inverse is $\Psi_j\circ\Gamma^{-1}_j\circ\Lambda$. 

\begin{ex}
Using Examples \ref{ex:main}, \ref{ex:crank-frobenius} and \ref{ex:cor}, the images of $(11,8,7,7,5,5,4,3,2,2)$ in the cases $j=0,3,5$ are respectively 
$$(12,9,7,6,5,5,4,2,1,1,1,1),(13,10,8,8,5,4,3,1,1,1)\text{ and }(12,10,8,7,6,5,5,1)\,.$$
In particular, in Theorem \ref{theo:crank-mex}, the partition $(11,8,7,7,5,5,4,3,2,2)$ with odd mex $1$ can be associated to the partition $(12,9,7,6,5,5,4,2,1,1,1,1)$ with non-negative crank $2$.
\end{ex}

\begin{ex} Here is a list of all the partitions of $9$ in $\M_0$ and their successive images by applying $\Phi_0$, $\Lambda_0$ and $\Gamma$.

$$
\begin{array}{|c|c|c|c|}
\hline
\la\in \M_0&\Phi_0(\la)\in \F_0&\Lambda_i(\Phi_0(\la))\in \C_{\leq 0}&\Gamma(\Lambda_i(\Phi_0(\la)))\in \C_{\geq 0}\\
\hline\hline
(9)&(9)&(8,1)&(8,1)\\
(7,2)&(8,1)&(7,1,1)&(6,2,1)\\
(6,3)&(6,3)&(5,2,1,1)&(5,3,1)\\
(5,4)&(5,4)&(4,3,1,1)&(4,3,1,1)\\
(5,2,2)&(7,1,1)&(6,1,1,1)&(4,2,2,1)\\
(4,3,2)&(4,3,2)&(3,2,2,1,1)&(4,4,1)\\
(3,3,3)&(4,4,1)&(3,3,1,1,1)&(3,3,3)\\
(3,2,2,2)&(6,1,1,1)&(5,1,1,1,1)&(2,2,2,2,1)\\
\hline
(6,2,1)&(5,3,1)&(4,2,1,1,1)&(3,3,2,1)\\
(5,2,1,1)&(4,3,1,1)&(3,2,1,1,1,1)&(4,3,2)\\
(4,2,2,1)&(3,3,2,1)&(2,2,2,1,1,1)&(3,2,2,2)\\
(4,2,1,1,1)&(3,3,1,1,1)&(2,2,1,1,1,1,1)&(5,2,2)\\
(2,2,2,2,1)&(5,1,1,1,1)&(4,1,1,1,1,1)&(5,4)\\
(2,2,2,1,1,1)&(4,1,1,1,1,1)&(3,1,1,1,1,1,1)&(6,3)\\
(2,2,1,1,1,1,1)&(3,1,1,1,1,1,1)&(2,1,1,1,1,1,1,1)&(7,2)\\
(2,1,1,1,1,1,1,1)&(2,1,1,1,1,1,1,1)&(1,1,1,1,1,1,1,1,1)&(9)\\
\hline
\end{array}\,.
$$ 
\end{ex}


\section*{Acknowledgement}

We would like to thank Brian Hopkins and Dennis Eichhorn. Their insightful comments on the early version of this paper help us to improve it.

This work was supported by the LABEX MILYON (ANR-10-LABX-0070) of Universit\'e de Lyon, within the program ''Investissements d'Avenir" (ANR-11-IDEX-0007) operated by the French National Research Agency (ANR).


\end{document}